\documentclass[11pt,a4paper,oneside,reqno]{amsproc}

\allowdisplaybreaks[3]

\usepackage[mathscr]{euscript}

\renewcommand{\mathcal}{\mathscr}

\usepackage{cite,url,hyperref}

\usepackage{graphicx}

\newcommand{\dd}{\mathrm{d}}
\renewcommand{\Tilde}{\widetilde}

\renewcommand{\Bar}{\overline}

\newcommand{\RR}{\mathbb{R}}
\newcommand{\CC}{\mathbb{C}}
\newcommand{\NN}{\mathbb{N}}
\newcommand{\cO}{\mathscr{O}}
\newcommand{\cH}{\mathcal{H}}

\newtheorem{theorem}{Theorem}
\newtheorem{prop}[theorem]{Proposition}
\newtheorem{lemma}[theorem]{Lemma}

\theoremstyle{definition}

\DeclareMathOperator{\spec}{spec}
\DeclareMathOperator{\dom}{\mathscr{D}}

\begin{document}

\title[]{On the Robin eigenvalues of the Laplacian\\ in the exterior of~a~convex polygon}

\author{Konstantin Pankrashkin}

\address{Laboratoire de math\'ematiques (UMR 8628), Universit\'e Paris-Sud,
B\^atiment 425, 91405 Orsay Cedex, France}

\email{konstantin.pankrashkin@math.u-psud.fr}
\urladdr{http://www.math.u-psud.fr/~pankrash/}

\begin{abstract}
Let $\Omega\subset \mathbb{R}^2$ be the exterior of a convex polygon whose side lengths are $\ell_1,\dots,\ell_M$.
For  $\alpha>0$, let $H^\Omega_\alpha$ denote the Laplacian in $\Omega$, $u\mapsto -\Delta u$,
with the Robin boundary conditions $\partial u/\partial\nu =\alpha u$,
where $\nu$ is the exterior unit normal at the boundary of~$\Omega$.
We show that, for any fixed $m\in\mathbb{N}$, the $m$th eigenvalue $E^\Omega_m(\alpha)$
of $H^\Omega_\alpha$ behaves as
\[
E^\Omega_m(\alpha)=-\alpha^2+\mu^D_m +\mathcal{O}\Big(\dfrac{1}{\sqrt\alpha}\Big)
\quad \mbox{as $\alpha$ tends to $+\infty$},
\]
where $\mu^D_m$ stands for the $m$th eigenvalue of the operator $D_1\oplus\dots\oplus D_M$ and
$D_n$ denotes the one-dimensional Laplacian $f\mapsto -f''$ on $(0,\ell_n)$ with the Dirichlet boundary conditions.

\bigskip

%\noindent {\sc Keywords:} eigenvalue asymptotics, Laplacian, Robin boundary condition, Dirichlet boundary condition

%\bigskip

%\noindent {\sc PACS:} 41.20.Cv, 02.30.Jr, 02.30.Tb

\end{abstract}

\maketitle

\section{Introduction}

\subsection{Laplacian with Robin boundary conditions}

Let $\Omega\subset \RR^d$, $d\ge2$, be a connected domain with a compact Lipschitz
boundary~$\partial\Omega$.. For $\alpha>0$,
let $H^\Omega_\alpha$ denote the Laplacian $u\mapsto-\Delta u$ in $\Omega$ with the Robin boundary conditions
$\partial u/\partial \nu=\alpha u$  at~$\partial\Omega$,
where $\nu$ stands for the outer unit normal. More precisely, $H_\alpha^\Omega$ is the self-adjoint operator in $L^2(\Omega)$ generated by the sesquilinear form
\[
h^\Omega_\alpha (u,u)=\iint_\Omega |\nabla u|^2\, \dd x -\alpha \int_{\partial\Omega}
|u|^2\,\dd\sigma,
\quad \dom (h^\Omega_\alpha)=W^{1,2}(\Omega).
\]
Here and below, $\sigma$ denotes the $(d-1)$-dimensional Hausdorff measure.

One checks in the standard way that the operator $H^\Omega_\alpha$
is semibounded from below. If $\Omega$ is bounded (i.e. $\Omega$
is an interior domain), then it has a compact resolvent, and we denote by $E^\Omega_m(\beta)$, $m\in\NN$, its eigenvalues taken according to their multiplicities
and enumerated in the non-decreasing order.
If $\Omega$ is unbounded (i.e. $\Omega$ is an exterior domain),
then the essential spectrum of $H^\Omega_\alpha$
coincides with $[0,+\infty)$, and the discrete spectrum consists of finitely many 
eigenvalues which will be denoted again
by $E^\Omega_m(\alpha)$, $m\in\{1,\dots,K_\alpha\}$,
and enumerated in the non-decreasing order taking into account the multiplicities.

We are interested in the behavior of the eigenvalues $E^\Omega_m(\alpha)$
for large $\alpha$. It seems that the problem was introduced
by Lacey, Ockedon, Sabina~\cite{lacey} when studying a reaction-diffusion system.
Giorgi and Smits~\cite{gs} studied a link to the theory of enhanced surface superconductivity.
Recently, Freitas and Krej\v ci\v r\'\i k~\cite{FK} and then Pankrashkin and Popoff~\cite{pp14}
studied the eigenvalue asymptotics in the context of the spectral optimization.

Let us list some available results. Under various assumptions
one showed the asymptotics of the form
\begin{equation}
   \label{eq-eig}
E^\Omega_m(\alpha)=-C_\Omega\alpha^2+o(\alpha^2)	\text{ as $\alpha$ tends to $+\infty$,}
\end{equation}
where $C_\Omega\ge 1$ is a constant depending on the geometric properties of~$\Omega$.
Lacey, Ockedon, Sabina~\cite{lacey} showed~\eqref{eq-eig} with $m=1$
for $C^4$ compact domains, for which $C_\Omega=1$,
and for triangles, for which $C_\Omega=2/(1-\cos\theta)$, where
$\theta$ is the smallest corner.
Lu and Zhu~\cite{luzhu} showed~\eqref{eq-eig} with $m=1$ and $C_\Omega=1$
for compact $C^1$ smooth domains, and
Daners and Kennedy~\cite{dan1} extended the result to any fixed $m\in\NN$.
Levitin and Parnovski~\cite{lp} showed \eqref{eq-eig} with $m=1$
for domains with piecewise smooth compact Lipschitz boundaries.
They proved, in particular, that if $\Omega$ is a curvilinear polygon
whose smallest corner is~$\theta$, then for $\theta<\pi$ there holds
$C_\Omega=2/(1-\cos\theta)$, otherwise $C_\Omega=1$.
Pankrashkin~\cite{kp13}
considered two-dimensional domains with a piecewise
$C^4$ smooth compact boundary and without convex corners, and it was shown
that $E^\Omega_1(\alpha)=-\alpha^2-\gamma \alpha+\cO(\alpha^{2/3})$, where
$\gamma$ is the maximum of the signed curvature at the boundary.
Exner, Minakov, Parnovski~\cite{emp14} showed that for compact $C^4$
smooth domains the same asymptotics $E^\Omega_m(\alpha)=-\alpha^2-\gamma \alpha+\cO(\alpha^{2/3})$ holds for any fixed $m\in\NN$. Similar results were obtained by Exner and Minakov~\cite{em}
for a class of two-dimensional domains with non-compact boundaries
and by Pankrashkin and Popoff~\cite{pp14} for $C^3$ compact
domains in arbitrary dimensions. Cakoni, Chaulet, Haddar~\cite{cch} studied the asymptotic
behavior of higher eigenvalues.

\subsection{Problem setting and the main result}

The computation of further terms in the eigenvalue asymptotics
needs more precise geometric assumptions. To our knowledge, such results
are available for the two-dimensional case only.
Helffer and Pankrashkin~\cite{HP} studied the tunneling effect for the
eigenvalues of a specific domain with two equal corners, and
Helffer and Kachmar~\cite{HK} considered the domains whose boundary curvature
has a unique non-degenerate maximum. The machinery of the both papers
is based on the asymptotic properties of the eigenfunctions:
it was shown that the eigenfunctions corresponding to the lowest eigenvalues
concentrate near the smallest convex corner at the boundary
or, if no convex corners are present, near the point
of the maximum curvature, and this is used to obtain
the corresponding eigenvalue asymptotics.

The aim of the present note is to consider a new class of two-dimensional domains~$\Omega$.
Namely, our assumption is as follows:
\[
\text{The domain $\RR^2\setminus\overline\Omega$
is a convex polygon (with straight edges).}
\]
Such domains are not covered by the above cited works: all the corners are non-convex,
and the curvature is constant on the smooth part of the boundary,
and it is not clear how the eigenfunctions are concentrated along the boundary.
We hope that our result will be of use  for the understanding of the role
of~non-convex corners.

In order to formulate the main result we need some notation.
Denote the vertices of the polygon $\RR^2\setminus\overline\Omega$
by $A_1,\dots, A_M\in\RR^2$, $M\ge 3$, and assume
that they are enumerated is such a way 
that the boundary $\partial\Omega$ is the union of the $M$ line segments
$L_n:=[A_n,A_{n+1}]$, $n\in\{1,\dots, M\}$, where we denote
$A_{M+1}:=A_1$, $A_{0}:=A_M$.
It is also assumed that there are no artificial vertices, i.e. that
$A_n\notin[A_{n-1},A_{n+1}]$ for any $n\in\{1,\dots,M\}$.

Furthermore, we denote by $\ell_n$ the length of the side $L_n$, 
and by $D_n$ the Dirichlet Laplacian $f\mapsto-f''$
on $(0,\ell_n)$ viewed as a self-adjoint operator in $L^2(0,\ell_n)$.
The main result of the present note is as follows:

\begin{theorem}\label{thm1}
For any fixed $m\in\mathbb{N}$ there holds
\[
E^\Omega_m(\alpha)=-\alpha^2+\mu^D_m +\cO\Big(\dfrac{1}{\sqrt\alpha}\Big) \text{ as $\alpha$ tends to $+\infty$,}
\]
where $\mu_m^D$ is the $m$th eigenvalue of the operator $D_1\oplus\dots\oplus D_M$.
\end{theorem}
The proof is based on the machinery proposed by Exner and Post~\cite{ep}
to study the convergence on graph-like manifolds.
Actually our construction appears to be quite similar to
that of Post~\cite{post} used to study decoupled waveguides.

We remark that due to the presence of non-convex corners
the domain of the operator $H^\Omega_\alpha$
contains singular functions and is not included
in~$W^{2,2}(\Omega)$, see e.g. Grisvard~\cite{grisv}.
This does not produce any difficulties as our approach
is purely variational and is
entirely based on the analysis of the sesqulinear form.

\section{Preliminaries}

\subsection{Auxiliary operators}

For $\alpha>0$, denote by $T_\alpha$ the following self-adjoint
operator in $L^2(\RR_+)$:
\[
T_\alpha v=-v'', \quad
\dom( T_\alpha)=\big\{
v\in W^{2,2}(\RR_+): \, v'(0)+\alpha v(0)=0
\big\}.
\]
It is well known that
\begin{equation}
     \label{eq-phi}
\spec T_\alpha=\{-\alpha^2\}\cup[0,+\infty),
\quad
\ker(T+\alpha^2)=\CC \varphi_\alpha,
\quad
\varphi_\alpha(s):=\dfrac{e^{-\alpha s}}{\sqrt{2\alpha}}.
\end{equation}
The sesqulinear form $t_\alpha$ for the operator $T_\alpha$ looks as follows:
\[
t_\alpha(v,v)=\int_0^\infty \big|v'(s)\big|^2\dd s-\alpha \big|v(0)\big|^2,
\quad \dom (t_\alpha)=W^{1,2}(\RR_+).
\]

\begin{lemma}\label{lem2}
For any $v\in W^{1,2}(\RR_+)$ there holds
\[
\int_{0}^\infty \big|v(s)\big|^2\dd s
- \bigg|\int_{0}^\infty \varphi_\alpha(s)v(s)\dd s\bigg|^2\\
\le
\dfrac{1}{\alpha^2} \bigg(
\int_0^\infty \big|v'(s)\big|^2\dd s -\alpha \big|v(0)\big|^2
+\alpha^2 \int_{0}^\infty \big|v(s)\big|^2\dd s
\bigg).
\]
\end{lemma}

\begin{proof}
Denote by $P$ the orthogonal projector on
$\ker (T_\alpha+\alpha^2)$ in $L^2(\RR_+)$, then by the spectral theorem we have
\[
t_\alpha(v,v)+\alpha^2 \|Pv\|^2=t_\alpha(v-Pv,v-Pv)\ge 0
\]
for any $v\in \dom (t_\alpha)$. As $\varphi_\alpha$ is normalized, there holds
\[
\bigg|\int_{0}^\infty \varphi_\alpha(x)v(x)\dd x\bigg|=\|P v\|,
\]
and we arrive at the conclusion.
\end{proof}

Another important estimate is as follows, see Lemmas~2.6 and~2.8 in~\cite{lp}:
\begin{lemma}\label{lem-lp}
Let $\Lambda\subset\RR^2$ be an infinite sector of opening $\theta\in(0,2\pi)$, then
for any $\varepsilon>0$ and any function $v\in W^{1,2}(\Lambda)$ there holds
\begin{equation}
   \label{eq-cc} 
\int_{\partial \Lambda} |v|^2\dd s
\le \varepsilon
\iint_{\Lambda}|\nabla v|^2\dd x + \dfrac{C_\theta}{\varepsilon}
\iint_{\Lambda}|v|^2\dd x
\quad\text{with}\quad
C_\theta=\begin{cases}
\dfrac{2}{1-\cos\theta}, & \theta\in(0,\pi),\\[\medskipamount]
\,1, & \theta\in [\pi,2\pi).
\end{cases}
\end{equation}
\end{lemma}

\subsection{Decomposition of $\Omega$}

Let us proceed with a decomposition of the domain $\Omega$
which will be used through the proof.  Let $n\in\{1,\dots,M\}$.
Denote by $S^1_n$ and $S^2_n$ the half-lines originating respectively
at $A_n$ and $A_{n+1}$, orthogonal to $L_n$ and contained in $\Omega$.
By $\Pi_n$ we denote the half-strip bounded by the half-lines $S^1_n$ and $S^2_n$
and the line segment $L_n$, and by $\Lambda_n$ we denote the infinite sector
bounded by the lines $S^2_{n-1}$ and $S^1_n$ and contained in $\Omega$.
The constructions are illustrated in~Figure~\ref{fig1}.
We note that the $2M$ sets $\Lambda_n$ and $\Pi_n$, $n\in\{1,\dots,M\}$, are
non-intersecting and that $\Bar \Omega=\bigcup_{n=1}^M\Bar\Lambda_n\mathop{\cup}\bigcup_{n=1}^M\Bar\Pi_n$. From Lemma~\ref{lem-lp} we deduce:
\begin{lemma}\label{lem4}
There exists a constant $C>0$ such that for any $\varepsilon>0$,
any $n\in \{1,\dots,M\}$ and any $v\in W^{1,2}(\Lambda_n)$ there holds
\[
\int_{\partial \Lambda_n} |v|^2\dd \sigma
\le C\varepsilon
\Big(
\iint_{\Lambda_n}|\nabla v|^2\dd x + \dfrac{1}{\varepsilon^2}
\iint_{\Lambda_n}|v|^2\dd x\Big).
\]
\end{lemma}
Furthermore, for each $n\in\{1,\dots,M\}$ denote by $\Theta_n$
the uniquely defined isometry $\RR^2\to\RR^2$
such that
\[
A_n=\Theta_n (0,0) \quad \text{and} \quad
\Pi_n=\Theta_n \big((0,\ell_n)\times\RR_+\big).
\]
We remark that due to the spectral properties of the above operator $T_\alpha$,
see~\eqref{eq-phi}, we have, for any $u\in W^{1,2}(\Pi_n)$,
\begin{multline*}
\int_0^{\ell_n}\int_0^\infty \Big|
\dfrac{\partial}{\partial s} u\big(\Theta_n(t,s)\big)\Big|^2 \dd s\, \dd t
-\alpha 
\int_0^{\ell_n} \Big|u\big(\Theta_n(t,s)\big)\Big|^2 \dd t
+\alpha^2
\int_0^{\ell_n}\int_0^\infty \Big|u\big(\Theta_n(t,s)\big)\Big|^2 \dd s \,\dd t\\
=\int_0^{\ell_n} \bigg(
\int_0^\infty\Big|\dfrac{\partial}{\partial s} u\big(\Theta_n(t,s)\big)\Big|^2 \dd s
-\alpha 
\Big|u\big(\Theta_n(t,0)\big)\Big|^2
+\alpha
\int_0^\infty\Big|u\big(\Theta_n(t,s)\big)\Big|^2 \dd s
\bigg)\dd t\ge0,
\end{multline*}
which implies, in particular,
\begin{multline}
             \label{eq-pin}
						0\le
\int_0^{\ell_n}\int_0^\infty \Big|
\dfrac{\partial}{\partial t} u\big(\Theta_n(t,s)\big)\Big|^2 \dd s\, \dd t\\
\begin{aligned}
\le\, & \int_0^{\ell_n}\int_0^\infty \Big|
\dfrac{\partial}{\partial t} u\big(\Theta_n(t,s)\big)\Big|^2 \dd s\, \dd t+
\int_0^{\ell_n}\int_0^\infty \Big|
\dfrac{\partial}{\partial s} u\big(\Theta_n(t,s)\big)\Big|^2 \dd s\, \dd t\\
&{}-\alpha 
\int_0^{\ell_n} \Big|u\big(\Theta_n(t,0)\big)\Big|^2 \dd t
+\alpha^2
\int_0^{\ell_n}\int_0^\infty \Big|u\big(\Theta_n(t,s)\big)\Big|^2 \dd s \,\dd t
\end{aligned}\\
=\iint_{\Pi_n} |\nabla u|^2\dd x -\alpha\int_{L_n} |u|^2\dd\sigma + \alpha^2 \iint_{\Pi_n}|u|^2\dd x.
\end{multline}

\begin{figure}
\centering
\includegraphics[height=50mm]{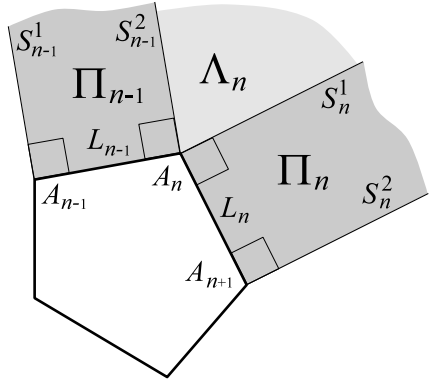}
\caption{Decomposition of the domain\label{fig1}}
\end{figure}

\subsection{Eigenvalues and identification maps}

We will use an eigenvalue estimate which is based on the max-min principle
and is just a suitable reformulation
of~Lemma~2.1 in~\cite{ep} or of~Lemma~2.2 in~\cite{post}:

\begin{prop}\label{prop6}
Let $B$ and $B'$ be non-negative self-adjoint operators
acting respectively in Hilbert space $\cH$ and $\cH'$ and
generated by sesqulinear form $b$ and $b'$.
Pick $m\in\NN$ and assume that
the operator $B$ has at least $m$ eigenvalues
$\lambda_1\le\dots\le\lambda_m<\inf\spec_\mathrm{ess} B$ and 
that the operator $B'$ has a compact resolvent.
If there exists a linear map  $J:\dom (b)\to\dom (b')$
(identification map) and two constants $\delta_1,\delta_2>0$ such that
$\delta_1\le (1+\lambda_m)^{-1}$ and that for any
$u\in\dom (b)$ there holds
\begin{align*}
\|u\|^2-\|Ju\|^2&\le \delta_1 \Big(  b(u,u)+\|u\|^2 \Big),\\
b'(Ju,Ju)-b(u,u)&\le \delta_2 \Big(  b(u,u)+\|u\|^2 \Big),
\end{align*}
then
\[
\lambda'_m\le \lambda_m + \dfrac{(\lambda_m\delta_1+\delta_2)(1+\lambda_m)}{1-(1+\lambda_m)\delta_1},
\]
where $\lambda'_m$ is the $m$th eigenvalue of the operator~$B'$.
\end{prop}

\section{Proof of Theorem~\ref{thm1}}

\subsection{Dirichlet-Neumann bracketing}

Consider the following sesqulinear form:
\begin{align*}
h^{\Omega,D}_\alpha(u,u)&=\sum_{n=1}^M
\iint_{\Lambda_n} |\nabla u|^2\dd x
+
\sum_{n=1}^M
\bigg(
\iint_{\Pi_n} |\nabla u|^2\dd x
-\alpha \int_{L_n} | u|^2\dd \sigma
\bigg),\\
\dom(h^{\Omega,D}_\alpha)&=\bigoplus_{n=1}^M W^{1,2}_0(\Lambda_n)\oplus
\bigoplus_{n=1}^M \Tilde W^{1,2}_0(\Pi_n),\\
\quad
\Tilde W^{1,2}_0(\Pi_n)&:=\big\{
f\in W^{1,2}(\Pi_n): f=0 \text{ at } S^1_n\mathop{\cup}S^2_n
\big\}.
\end{align*}
and denote by  $H^{\Omega,D}_\alpha$
the associated self-adjoint operator in $L^2(\Omega)$.
Clearly, the form  $h^{\Omega,D}_\alpha$ is a restriction of the initial form
$h^\Omega_\alpha$, and due to the max-min principle we have
\[
E_m^\Omega(\alpha)\le E_m^{\Omega,D}(\alpha),
\] where
$E_m^{\Omega,D}(\alpha)$ is the $m$th eigenvalue of $H^{\Omega,D}_\alpha$ (as soon at
it exists). On the other hand, we have the decomposition
\[
H^{\Omega,D}_\alpha= \bigoplus_{n=1}^M\big(-\Delta^D_n\big)\oplus 
\bigoplus_{n=1}^M G^D_{n,\alpha},
\]
where $(-\Delta^D_n)$ is the Dirichlet Laplacian in $L^2(\Lambda_n)$
and $G^D_{n,\alpha}$ is the self-adjoint operator
in $L^2(\Pi_n)$ generated by the sesquilinear form
\[
g^D_{n,\alpha}(u,u)=\iint_{\Pi_n} |\nabla u|^2\dd x -\alpha\int_{L_n}|u|^2 \dd \sigma,
\quad \dom (g^D_{n,\alpha})= \Tilde W^{1,2}_0(\Pi_n).
\]
Consider the following unitary maps:
\[
U_n:L^2(\Pi_n)\to L^2\big((0,\ell_n)\times\RR_+\big), \quad
U_n f:= f\circ \Theta_n,
\quad
n\in\{1,\dots,M\},
\]
then it is straightforward to check that
$U_n G^D_{n,\alpha} U^*_n= D_n \otimes 1 + 1\otimes T_\alpha$.
As the operators $(-\Delta^D_n)$ are non-negative,
it follows that 
$\spec_\text{ess} H^{\Omega,D}_\alpha=[0,+\infty)$
and that $E_m^{\Omega,D}(\alpha)=-\alpha^2+\mu^D_m$, which
gives the majoration
\begin{equation}
   \label{eq-minor}
	E_m^\Omega(\alpha)\le -\alpha^2+\mu^D_m
	\end{equation}
for all $m$ with $\mu^D_m< \alpha^2$.
In particular, the inequality \eqref{eq-minor} holds for any fixed $m$ as $\alpha$
tends to~$+\infty$.

Similarly, introduce the following sesquilinear form:
\begin{align*}
h^{\Omega,N}_\alpha(u,u)&=\sum_{n=1}^M
\iint_{\Lambda_n} |\nabla u|^2\dd x
+
\sum_{n=1}^M
\bigg(
\iint_{\Pi_n} |\nabla u|^2\dd x
-\alpha \int_{L_n} | u|^2\dd \sigma
\bigg),\\
\dom(h^{\Omega,N}_\alpha)&=\bigoplus_{n=1}^M W^{1,2}(\Lambda_n)\oplus
\bigoplus_{n=1}^M W^{1,2}(\Pi_n),
\end{align*}
and denote by  $H^{\Omega,N}_\alpha$ the associated self-adjoint operator in $L^2(\Omega)$.
Clearly, the initial form  $h^{\Omega}_\alpha$ is a restriction of the form
$h^{N,\Omega}_\alpha$, and due to the max-min principle we have
\[
E_m^{\Omega,N}(\alpha)\le E_m^\Omega(\alpha),
\]
where
$E_m^{\Omega,N}(\alpha)$ is the $m$th eigenvalue of $H^{\Omega,N}_\alpha$, and
the inequality holds for those $m$ for which $E_m^\Omega(\alpha)$ exists.
On the other hand, we have the decomposition
\[
H^{\Omega,N}_\alpha= \bigoplus_{n=1}^M\big(-\Delta^N_n\big)\oplus 
\bigoplus_{n=1}^M G^N_{n,\alpha},
\]
where $(-\Delta^N_n)$ denotes the Neumann Laplacian in $L^2(\Lambda_n)$
and $G^N_{n,\alpha}$ is the self-adjoint operator
in $L^2(\Pi_n)$ generated by the sesquilinear form
\[
g^N_{n,\alpha}(u,u)=\iint_{\Pi_n} |\nabla u|^2\dd x -\alpha\int_{L_n}|u|^2 \dd \sigma,
\quad \dom (g^N_{n,\alpha})= W^{1,2}(\Pi_n).
\]
There holds $U_n G^N_{n,\alpha} U_n^*= N_n \otimes 1 + 1\otimes T_\alpha$,
where $N_n$ is the operator $f\mapsto -f''$ on $(0,\ell_n)$
with the Neumann boundary condition viewed as a self-adjoint operator
in the Hilbert space $L^2(0,\ell_n)$, $n\in\{1,\dots,M\}$.
The operators $(-\Delta^N_n)$ are non-negative, and we have
$\spec_\text{ess} H^{\Omega,N}_\alpha=[0,+\infty)$
and  $E_m^{\Omega,N}(\alpha)=-\alpha^2+\mu^N_m$,
where $\mu^N_m$ is the $m$th eigenvalue of the operator
$N_1\oplus\dots\oplus N_M$. Thus we obtain the minorations
\begin{equation}
   \label{eq-minor2}
	H^\Omega_\alpha\ge -\alpha^2 \text{ and }
	E_m^\Omega(\alpha)\ge -\alpha^2+\mu^N_m,
	\end{equation}
which holds for any fixed $m$ as $\alpha$ tends to $+\infty$.
By combining the inequalities \eqref{eq-minor} and \eqref{eq-minor2} we obtain also the rough estimate
\begin{equation}
      \label{eq-rough}
E_m^\Omega(\alpha)=-\alpha^2+\cO(1) \quad
\text{for any fixed $m$ and for $\alpha$ tending to $+\infty$.}
\end{equation}

\subsection{Construction of an identification map}
In order to conclude the proof of Theorem~\ref{thm1} we are going
to apply Proposition~\ref{prop6} to the operators
\[
B=H^\Omega_\alpha+\alpha^2, \quad B'=D_1\oplus\dots\oplus D_n,
\]
which will allow us to obtain another inequality between the quantities
\[
\lambda_m=E^\Omega_m(\alpha)+\alpha^2, \quad
\lambda'_m=\mu_m^D.
\]
Note that for any fixed $m\in\NN$
one has $\lambda_m=\cO(1)$ for  large $\alpha$, see~\eqref{eq-rough}. Therefore, it is sufficient
to construct an identification map $J=J_\alpha$ as in Proposition~\ref{prop6}
with $\delta_1+\delta_2=\cO(\alpha^{-1/2})$.
Recall that the respective forms $b$ and $b\,'$ in our case are given by
\begin{align*}
b(u,u)&=h^\Omega_\alpha(u,u)+\alpha^2\|u\|^2,\quad
\dom(b)=\dom(h^\Omega_\alpha)=W^{1,2}(\Omega),\\
b\,'(f,f)&=\sum_{n=1}^M \int_0^{\ell_n}\big|f'_n(t)\big|^2\dd t,
\quad
\dom (b\,')=\big\{f=(f_1,\dots,f_M): f_n \in W^{1,2}_0(0,\ell_n)\big\}.
\end{align*}
Here and below, by $\|u\|$ we mean the usual norm in $L^2(\Omega)$.
The positivity of~$b\,'$ is obvious, and the positivity of $b$ follows from~\eqref{eq-minor2}.

Consider the maps
\[
P_{n,\alpha}: W^{1,2}(\Pi_n)\to L^2(0,\ell_n),
\quad
(P_{n,\alpha} u)(t)=\int_0^\infty \varphi_\alpha(s) u\big(\Theta_n(t,s)\big)\dd s,
\quad n\in\{1,\dots,M\}.
\]
If $u\in W^{1,2}(\Omega)$, then
 $u\in W^{1,2}(\Pi_n)$ for any $n\in\{1,\dots,M\}$, and one can
estimate, using the Cauchy-Schwarz inequality,
\[
\begin{aligned}
\big|(P_{n,\alpha} u)(0)\big|^2+\big|(P_{n,\alpha} u)(\ell_n)\big|^2&
\le \int_0^\infty \big|u\big(\Theta_n(0,s)\big)\big|^2\dd s
+
 \int_0^\infty \big|u\big(\Theta_n(\ell_n,s)\big)\big|^2\dd s\\
&=\int_{S^1_n} |u|^2 \dd\sigma+\int_{S^2_{n}} |u|^2 \dd\sigma.
\end{aligned}
\]
As $S^2_{n-1}\mathop{\cup} S^1_n=\partial \Lambda_n$, we can use
 Lemma~\ref{lem4} with $\varepsilon=\alpha^{-1}$, which gives
\begin{multline}
           \label{eq-nce2}
\sum_{n=1}^M \Big(\big|(P_{n,\alpha} u)(0)\big|^2+\big|(P_{n,\alpha} u)(\ell_n)\big|^2\Big)
\le \sum_{n=1}^M
\Big(
\int_{S^1_n} |u|^2 \dd\sigma+\int_{S^2_{n}} |u|^2 \dd\sigma
\Big)
\\
= \sum_{n=1}^M \int_{\partial \Lambda_n} |u|^2\dd\sigma
\le \dfrac{C}{\alpha} \sum_{n=1}^M  \bigg(\iint_{\Lambda_n} |\nabla u|^2\dd x
+\alpha^2\iint_{\Lambda_n} |u|^2\dd x \bigg).
\end{multline}

For each $n\in\{1,\dots,M\}$ introduce a map
\[
\pi_n:(0,\ell_n)\to\{0,\ell_n\},
\quad
\pi_n(t)=0 \text{ for } t< \dfrac{\ell_n}{2},\quad
\pi_n(t)=\ell_n \text{ otherwise,}
\]
and pick a function
$\rho_n\in C^\infty\big([0,\ell_n]\big)$
with $\rho_n(0)=\rho_n(\ell_n)=1$ and $\rho_n\Big(\dfrac{\ell_n}{2}\Big)=0$.

Finally, define
\[
J_\alpha:W^{1,2}(\Omega)\to \bigoplus_{n=1}^M L^2(0,\ell_n),
\quad
(J_\alpha u)_n(t)=(P_{n,\alpha} u) (t)- (P_{n,\alpha} u) \big(\pi_n(t)\big)\rho_n(t).
\]
We remark that $(J_\alpha u)_n \in W^{1,2}_0(0,\ell_n)$ for any $u\in W^{1,2}(\Omega)$
and $n\in\{1,\dots,M\}$, i.e. $J_\alpha$ maps $\dom(b)$
into $\dom(b\,')$ and will be used as an identification map.

\subsection{Estimates for the identification map}

Take any $\delta>0$. Using the inequality
\[
(a_1+a_2)^2\ge (1-\delta) a_1^2 -\dfrac{1}{\delta}\, a_2^2,
\quad a_1,a_2\ge0,
\]
we estimate
\begin{align*}
\|u\|^2-\|J_\alpha u\|^2
=\,&\sum_{n=1}^M  \iint_{\Lambda_n} |u|^2\dd x\\
&\qquad+\sum_{n=1}^M
\Big(\iint_{\Pi_n}|u|^2\dd x
-\int_{0}^{\ell_n}
\Big|
\big(P_{n,\alpha}u\big)(t) - \big(P_{n,\alpha}u\big)\big(\pi(t)\big)\rho(t)
\Big|^2\dd t
\Big)\\
\le\, & \sum_{n=1}^M  \iint_{\Lambda_n} |u|^2\dd x
+\sum_{n=1}^M
\bigg(\iint_{\Pi_n}|u|^2\dd x
-(1-\delta)\int_{0}^{\ell_n}
\big|\big(P_{n,\alpha}u\big)(t)\big|^2\dd t\\
&\qquad+\dfrac{1}{\delta}
\int_{0}^{\ell_n}
\big|\big(P_{n,\alpha}u\big)\big(\pi(t)\big)\rho(t)\big|^2\dd t
\bigg)\\
=\, &
\sum_{n=1}^M  \iint_{\Lambda_n} |u|^2\dd x
+
\sum_{n=1}^M
\bigg(
\iint_{\Pi_n} |u|^2\dd x
-\int_0^{\ell_n}
\big|\big(P_{n,\alpha}u\big)(t)\big|^2\dd t
\bigg)\\
&\qquad
+\delta \sum_{n=1}^M  \int_{0}^{\ell_n}
\big|\big(P_{n,\alpha}u\big)(t)\big|^2\dd t+\dfrac{1}{\delta}
\sum_{n=1}^M 
\int_{0}^{\ell_n}
\Big|\big(P_{n,\alpha}u\big)\big(\pi(t)\big)\rho_n(t)\Big|^2\dd t\\
&=:I_1+I_2+I_3+I_4.
\end{align*}
We have the trivial inequality
\[
I_1\le \dfrac{1}{\alpha^2}\sum_{n=1}^M\Big( \iint_{\Lambda_n} |\nabla u|^2\dd x+\alpha^2\iint_{\Lambda_n} |u|^2\dd x\Big).
\]
To estimate the term $I_2$ we use Lemma~\ref{lem2} and then~\eqref{eq-pin}:
\begin{align*}
I_2 
&=
\sum_{n=1}^M  \int_0^{\ell_n}
\Big(
\int_0^\infty \big|u\big(\Theta_n(t,s)\big)\big|^2\dd s
-\Big|\int_0^\infty
\varphi_\alpha(s) u\big(\Theta_n(t,s)\big)\dd s
\Big|^2
\Big)\dd t\\
&\le \dfrac{1}{\alpha^2}
\sum_{n=1}^M \int_0^{\ell_n}
\Big(
\int_0^\infty \Big| \dfrac{\partial }{\partial s}u\big(\Theta_n(t,s)\big)\Big|^2\dd s
-\alpha \big|u\big(\Theta_n(t,0)\big)\big|^2\\
&\qquad+\alpha^2 \int_0^\infty \big|u\big(\Theta_n(t,s)\big)\big|^2\dd s
\Big)\dd t\\
&\le \dfrac{1}{\alpha^2}\sum_{n=1}^M\Big( \iint_{\Pi_n} |\nabla u|^2\dd x -\int_{L_n}|u|^2\dd \sigma + \alpha^2\iint_{\Pi_n} |u|^2\dd x \Big),
\end{align*}
which gives
\[
I_1+I_2\le\dfrac{1}{\alpha^2}\Big(
h^\Omega_\alpha(u,u)+\alpha^2\|u\|^2
\Big).
\]
Furthermore, with the help of the Cauchy-Schwarz inequality we have
\[
I_3\le
\delta\sum_{n=1}^M \int_{0}^{\ell_n}\int_0^\infty \Big|u\big(\Theta_n(t,s)\big)\Big|^2\dd s \,\dd t
=\delta\sum_{n=1}^M \iint_{\Pi_n} |u|^2\dd x\le \delta \|u\|^2,
\]
To estimate the last term $I_4$ we introduce the constant
\[
R:=\max \Big\{\int_{0}^{\ell_n}\big|\rho_n(t)\big|^2\dd t: \, n\in\{1,\dots,M\}\Big\},
\]
then, using first the estimate~\eqref{eq-nce2} and then the inequality~\eqref{eq-pin},
\begin{align*}
I_4 &\le \dfrac{R}{\delta} \sum_{n=1}^M \sup_{t\in(0,\ell_n)}
 \Big|\big(P_{n,\alpha} u \big)\big(\pi_n(t)\big)\Big|^2\\
&\le \dfrac{R}{\delta}
\sum_{n=1}^M \Big(\big|(P_{n,\alpha} u)(0)\big|^2+\big|(P_{n,\alpha} u)(\ell_n)\big|^2\Big)\\
&\le \dfrac{RC}{\delta \alpha} \sum_{n=1}^M \bigg(\iint_{\Lambda_n} |\nabla u|^2\dd x
+\alpha^2\iint_{\Lambda_n} |u|^2\dd x \bigg)\\
&\le\dfrac{RC}{\delta \alpha} \bigg[\sum_{n=1}^M \Big(\iint_{\Lambda_n} |\nabla u|^2\dd x
+\alpha^2\iint_{\Lambda_n} |u|^2\dd x\Big)\\
&\qquad +\sum_{n=1}^M\Big(
\iint_{\Pi_n} |\nabla u|^2\dd x -\int_{L_n}|u|^2\dd \sigma + \alpha^2\iint_{\Pi_n} |u|^2\dd x
\Big)
\bigg]\\
&=\dfrac{RC}{\delta\alpha}\Big(h^\Omega_\alpha(u,u)+\alpha^2\|u\|\Big).
\end{align*}
Choosing $\delta=\alpha^{-1/2}$ and summing up the four terms we see that
\[
\|u\|^2-\|J_\alpha u\|^2\le
\dfrac{c_1}{\sqrt{\alpha}}\Big(h^\Omega_\alpha(u,u)+\alpha^2\|u\|^2 +\|u\|^2\Big)
\equiv
\dfrac{c_1}{\sqrt{\alpha}}\Big(b(u,u) +\|u\|^2\Big).
\]
with a suitable constant $c_1>0$.

Now we need to compare $b\,'(J_\alpha u,J_\alpha u)$ and $b(u,u)$.
Take $\delta\in(0,1)$ and use the inequality
\[
(a_1+a_2)^2\le (1+\delta)a_1^2+ \dfrac{2}{\delta}\, a_2^2, \quad a_1,a_2\ge 0, 
\]
then
\begin{multline}
          \label{eq-forms}
b\,'(J_\alpha u,J_\alpha u)- b(u,u) =\sum_{n=1}^M \int_0^{\ell_n}
\Big|(P_{n,\alpha}u)' - \rho\,'_n \big[(P_{n,\alpha} u)\circ \pi_n\big]\Big|^2\dd t
- \Big(
h^\Omega_\alpha(u,u)+\alpha^2\|u\|^2
\Big)\\
\begin{aligned}
\le&
(1+\delta)\sum_{n=1}^M \int_0^{\ell_n}
\Big|(P_{n,\alpha}u)'\Big|^2\dd t+\dfrac{2}{\delta}
\sum_{n=1}^M \int_0^{\ell_n}
\Big|\rho\,'_n \big[(P_{n,\alpha} u)\circ \pi_n\big]\Big|^2\dd t\\
&\quad{}-\sum_{n=1}^M \bigg(
\iint_{\Lambda_n} |\nabla u|^2\dd x+\alpha^2\iint_{\Lambda_n} |u|^2\dd x
\bigg)\\
&\quad{}-\sum_{n=1}^M\Big(
\iint_{\Pi_n} |\nabla u|^2\dd x -\int_{L_n}|u|^2\dd \sigma + \alpha^2\iint_{\Pi_n} |u|^2\dd x
\Big)\\
&\le
(1+\delta)\sum_{n=1}^M \int_0^{\ell_n}
\Big|(P_{n,\alpha}u)'\Big|^2\dd t
+\dfrac{2}{\delta}
\sum_{n=1}^M \int_0^{\ell_n}
\Big|\rho\,'_n \big[(P_{n,\alpha} u)\circ \pi_n\big]\Big|^2\dd t\\
&\quad{}-\sum_{n=1}^M\Big(
\iint_{\Pi_n} |\nabla u|^2\dd x -\int_{L_n}|u|^2\dd \sigma + \alpha^2\iint_{\Pi_n} |u|^2\dd x
\Big).
\end{aligned}
\end{multline}
Using first the Cauchy-Schwarz inequality and then the inequality~\eqref{eq-pin}
we have
\begin{multline*}
\int_0^{\ell_n}\big|(P_{n,\alpha}u)'\big|^2\dd t
\le \int_0^{\ell_n}\int_0^\infty \Big|
\dfrac{\partial}{\partial t} u\big(\Theta_n(t,s)\big)\Big|^2 \dd s\, \dd t\\
\le
\iint_{\Pi_n} |\nabla u|^2\dd x -\int_{L_n}|u|^2\dd \sigma + \alpha^2\iint_{\Pi_n} |u|^2\dd x.
\end{multline*}
Substituting the last inequality into \eqref{eq-forms} we arrive at
\begin{equation}
    \label{eq-loc2}
\begin{aligned}		
b\,'(J_\alpha u,J_\alpha u)- b(u,u)&\le \delta
\sum_{n=1}^M \Big(
\iint_{\Pi_n} |\nabla u|^2\dd x -\int_{L_n}|u|^2\dd \sigma + \alpha^2\iint_{\Pi_n} |u|^2\dd x
\Big)\\
&\quad+\dfrac{2}{\delta}
\sum_{n=1}^M \int_0^{\ell_n}
\Big|\rho'_n \big[(P_{n,\alpha} u)\circ \pi_n\big]\Big|^2\dd t. 
\end{aligned}
\end{equation}
Furthermore, using the constant
\[
R\,':=\max \Big\{\int_{0}^{\ell_n}\big|\rho'_n(t)\big|^2\dd t: \, n\in\{1,\dots,M\}\Big\},
\]
and the inequality \eqref{eq-nce2} we have
\[
\begin{aligned}
\sum_{n=1}^M \int_0^{\ell_n}
\Big|\rho'_n \big[(P_{n,\alpha} u)\circ \pi_n\big]\Big|^2\dd t&\le
R\,'\sum_{n=1}^M \sup_{t\in(0,\ell_n)} \big|(P_{n,\alpha} u) \big(\pi_n(t)\big)\big|^2\\
&\le
R\,'\sum_{n=1}^M \Big(\big|(P_{n,\alpha} u)(0)\big|^2+\big|(P_{n,\alpha} u)(\ell_n)\big|^2\Big)\\
&\le 
\dfrac{R\,' C}{\alpha}\sum_{n=1}^M \bigg(\iint_{\Lambda_n} |\nabla u|^2\dd x
+\alpha^2\iint_{\Lambda_n} |u|^2\dd x \bigg).
\end{aligned}
\]
The substitution of this inequality into~\eqref{eq-loc2}
and the choice $\delta=\alpha^{-1/2}$ lead then to
\[
b\,'(Ju,Ju)- b(u,u)
\le \dfrac{c_2}{\sqrt\alpha} \Big(h^\Omega_\alpha(u,u)+\alpha^2\|u\|^2\Big)
\le \dfrac{c_2}{\sqrt\alpha} \Big(b(u,u)+\|u\|^2\Big).
\]
with a suitable constant $c_2>0$. By Proposition~\ref{prop6}, for any fixed $m\in\NN$
and for large $\alpha$ we have the estimate
$\mu^D_m\le E^\Omega_m(\alpha) +\alpha^2 +\cO(\alpha^{-1/2})$.
The combination with~\eqref{eq-minor} gives the result.

\section*{Acknowledgments}
The work was partially supported by ANR NOSEVOL (ANR 2011 BS01019 01)
and GDR Dynamique quantique (GDR CNRS 2279 DYNQUA).

\end{document}